\documentclass{amsart}

\usepackage{amssymb}

\newtheorem{theorem}{Theorem}[section]

\newtheorem{lemma}[theorem]{Lemma}

\theoremstyle{definition}
\newtheorem{definition}[theorem]{Definition}
\newtheorem{example}[theorem]{Example}
\theoremstyle{remark}

\numberwithin{equation}{section}
\newcommand{\R}{\mathbb{R}}  
\newcommand{\Z}{\mathbb{Z}}  
\newcommand{\N}{\mathbb{N}}  
\newcommand{\Nb}{\mathfrak{N}}  
\newcommand{\cuboid}{\mathop{cuboid}}
\newcommand{\card}{\mathop{\mathrm{card}}}
\newcommand{\sll}{\mathop{<\kern-4pt<}}

\begin{document}

\title{Local Finiteness of Infinite Neighbor Complexes}
\author{James J.~Madden}
\address{Louisiana State University, Baton Rouge}
\email{madden@math.lsu.edu, jamesjmadden@gmail.com}
\date{August 15, 2016}

\begin{abstract} 
We show that if the neighbor complex, as defined by H.~Scarf, of an infinite subset of $\Z^n$ has finite dimension, then each vertex has finitely many neighbors.   \end{abstract}

\maketitle

\section{Background and Motivation}

We use the following notation.  For $a, b\in \R^n$, $a\leq b$ means that $a$ is less than or equal to $b$ in every coordinate and $a\sll b$ means that $a$ is \emph{strictly} less than $b$ in every coordinate, i.e., $\pi_i(a)<\pi_i(b)$ for $i=1,\ldots,n$.  The coordinatewise maximum of $a$ and $b$ is written $a\vee b$, and the coordinatewise minimum is written $a\wedge b$.  If $B\subseteq \R^n$, then $\vee B$ denotes the coordinatewise supremun in $(\R\cup\{\pm\infty\})^n$, i.e., $\pi_i(\vee B) = \sup\{\,\pi_i(b) \mid b\in B\,\}$.  

\begin{definition} Let $A$ be a subset of $\R^n$.  The \textit{neighbor complex of $A$}, denoted $\Nb(A)$, is the set of all $B\subseteq A$ such that there is no $a\in A$ satisfying $a\sll \vee B$.  If $\{a, a'\}\in \Nb(A)$, we say \textit{$a$ and $a'$ are $A$-neighbors}.
\end{definition}

If $B'\subseteq B\in \Nb(A)$, then $B'\in \Nb(A)$.  Thus $\Nb(A)$ is an abstract simplicial complex.  This complex was introduced by Herbert Scarf in \cite{Sc} as a tool in integer programming.   Scarf was particularly interested in studying $\Nb(G)$, when $G$ is a discrete generic subgroup of $\R^n$ and $G\cap\R_{\geq0}^n=\{0\}$. The meaning of ``generic'' is discussed in the next paragraph.

In the present work, we shall say that $A\subseteq \R^n$ is \emph{generic} if for any $A$-neighbors $a, a'\in A$,  $\pi_i(a)\not=\pi_i(a')$ for all $i = 1, \ldots,n$. In fact, there are several variants of the notion of generic; see \cite{MM} for a discussion.  The sense of generic used by Scarf was stronger than the one we use here, but since our main theorem makes no reference to genericity, there is no reason to say any more about this here.   

B\'ar\'any, Howe, Scarf and Shallcross \cite{BHS}, \cite{BSS} determined the topology of $\Nb(G)$ when $G$ is a group of the kind studied by Scarf.  Scarf observed that if $m$ is any positive interger, then there are generic subgroups $G\subseteq \Z^4$ in which every element has more than $m$ $G$-neighbors. Bounds for the number of neighbors can be given in terms of the size of the integers required to give a basis for $A$; see \cite{Sh}.  

At present, one of the chief motivations for studying $\Nb(A)$ comes from algebra.  Let $S$ be the algebra of polynomials in $n$ variables over a field and suppose that $A$ consists of the exponent vectors of a minimal monomial generating set of a monomial ideal $I\subseteq S$.  Bayer, Peeva and Sturmfels,  \cite{BPS} showed that if $A$ is generic, then $\Nb(A)$ supports a minimal free resolution of $I$.   Recently,  Olteanu and Welker \cite{OW} introduced new combinatorial methods to study $\Nb(A)$ when $A$ comes from a monomial ideal in this way.  They showed  that even when $A$ is not generic, $\Nb(A)$ supports a free resolution, though not a minimal one. Of course in these applications, $A$ is finite.  

Bayer and Sturmfels \cite{BS} generalized the method of \cite{BPS} to find combinatorial resolutions of binomial ideals.  In their work, they use the neighbor complex of a generic subgroup $G\subseteq \Z^n$.  The complex $\Nb(G)$ admits an action by $G$.  They display a $\Z^n/G$-graded resolution of the lattice ideal $I_G:=\langle X^{\gamma^+}-X^{\gamma^-}\mid \gamma\in G\rangle\subseteq S$ supported by the quotient complex $\Nb(G)/G$.    

McGuire \cite{M} considered the case when $A$ is the union of finitely many cosets of a subgroup $G\subset\Z^n$, and he generalized the method of \cite{BS} to resolve ideals generated by binomials as well as monomials.  Here the quotient complex $\Nb(A)/G$ also comes into play.  

In algebraic applications involving infinite $A$, such as those mentioned in the previous paragraph, it is important to know that every vertex of $\Nb(A)$ has finitely many neighbors, or as we say, $\Nb(A)$ is locally finite, cf. the discussion in \cite{MS}, page 178.   The purpose of the present paper is to provide a very general criterion for the local finiteness of $\Nb(A)$.

\section{Main Theorem}\label{sec2}

We use the usual definitions of dimension for simplicial complexes, namely, if $B\in \Nb(A)$, then $\dim B:=\card B-1$.  If there is $d\in \N$ such that  $\dim B\leq d$ for all $B\in \Nb(A)$ \emph{and} $\dim B= d$ for at least one $B\in \Nb(A)$, then we say $\dim \Nb(A)= d$.   Note that $\dim \Nb(A) = d$ implies that for each $B\in \Nb(A)$, $\card \left(A\cap (\vee B+\R^n_{\leq 0})\right)\leq d+1$. 

\begin{definition}We say that a simplicial complex is \emph{locally finite} if each vertex is contained in finitely many simplices.\end{definition}

Clearly, a complex is finite if and only if each vertex is in finitely many 1-simplices.  In particular, $\Nb(A)$ is locally-finite  if and only if every $a \in A$ has finitely many $A$-neighbors.  In general, a $d$-dimensional simplicial complex need not be locally finite.   For example, consider the $1$-dimensional complex with vertex set $\N$ and edge set consisting of all $\{0,j\}$ with $0\not=j\in \N$.

\begin{theorem}
If $A\subseteq \Z^n$ and $\dim \Nb(A)=d$ for some $d\in \N$, then $\Nb(A)$ is locally finite.\end{theorem}

The following examples show that neither hypothesis may be omitted.

\begin{example}  Let $A=\{\,(n, 0)\in \R^2\mid n\in \N\,\}$.   Every simplex of $\Nb(A)$ is finite.  But every point of $A$ is an $A$-neighbor of every other point, so $\Nb(A)$ is not locally finite.   In this example, $A\subseteq \Z^n$ but $\dim\Nb(A) =\infty$. \end{example}

\begin{example} Let $A=\{a_0,a_1, \ldots\}$, where $a_0=(0, 0, 1)$ and $a_i=(i, \frac{1}{i}, \frac{i-1}{i})$.  The maximal simplices of $\Nb(A)$ are: $\{a_0, a_i, a_{i+1}\}\in \Nb(A)$, for $i=1,2,\ldots$. Since $a_0$ has infinitely many neighbors, $\Nb(A)$ is not locally finite.  In this example, $\dim\Nb(A)=2$ but $A\not\subseteq \Z^n$.
\end{example}

In \cite{MM}, the authors show that  if $A\subseteq \Z^n$ and $A$ is generic, then $\Nb(A)$ is locally finite.  The present theorem includes this result, for $A\subseteq \R^n$ is generic if and only if: for all $B\in \Nb(A)$, $\vee B+\R_{\leq 0}^n$ has \textit{at most} one point of $A$ on each face.  Thus, if $A\subseteq \R^n$ is generic, then $\dim \Nb(A)<n$.

\section{Proof}

If $P$ is a poset with order $\leq_P$, $s\in P$ and $S \subseteq P$, then we let $$\downarrow\negthinspace(s,S):=\{ \,u\in S\mid u\leq_P s\,\}.$$

\begin{lemma}\label{dlem} Let $S\subseteq \N^n$.    Then for any $k\in\N$, $\{\,s\in S\mid \card \downarrow\negthinspace(s,S)\leq k+1\,\}$ is finite.  \end{lemma} 

\begin{proof}Let ${S_0}$ denote the set of minimal element of $S$. Dickson's Lemma states that $S_0$ is finite.  Assuming that $S_0, S_1, \ldots S_k$ have been defined, let
$$S_{(k)} = \bigcup\{\,S_i\mid 0\leq i\leq k\,\},$$
and let ${S_{k+1}}:=$ the set of minimal elements of $S\setminus S_{(k)}$.    
 Repeated applications of Dickson's Lemma show that $S_k$ is finite for each $k\in \N$, and hence $S_{(k)}$ is finite for each $k$. Evidently $\{\,s\in S\mid \card \downarrow\negthinspace(s,S)\leq k+1\,\}\subseteq S_{(k)}$. \end{proof}

For $a,b\in \R^n$, let $\cuboid(a,b)$ denote the set of $x\in \R^n$ that are coordinatewise between $a$ and $b$. i.e., 
 $$cuboid(a,b) = \{x\in \R^n\mid a\wedge b \leq x \leq a\vee b\,\}.$$
 Let $P$ be any orthant $\R^n$.  We view $P$ as a poset with order $\leq_{_P}$, where $a\leq_P b$ means $b-a\in P$.  With this order, $P$ is order-isomorphic to $\langle \R^n, \leq\rangle$.  Note that if $a\in P$, then $$cuboid(0,a) =  \{x\in \R^n\mid 0 \leq_{_P} x \leq_{_P} a\,\} =\downarrow\negthinspace(a, P).$$ 
 For any $A\subseteq \R^n$, any orthant $P$ and any $k\in \N$,  let 
$$A^P_{[k]}:=\{\,a\in A\cap P\mid\card \downarrow\negthinspace(a, A\cap P)\leq k+1\,\}.$$   
 
We now complete the proof of the Theorem.  Assume $A\subseteq \Z^n$ and $\dim \Nb(A)=d\in \N$.  
We want to show that each $a\in A$ has finitely many $A$-neighbors.   By translation, we may assume that $a=0$.   It suffices to show that $0$ has at most finitely many $A$-neighbors in each orthant.  So, pick any orthant $P$.  Our strategy is to show that every $A$-neighbor of $0$ in $P$ belongs to $A^P_{[d]}$, which we know to be finite by Lemma \ref{dlem}.  Suppose $b\in P\cap A$ is an $A$-neighbor of $0$. Then $\cuboid(0,b)\subseteq (b\vee 0)+\R^n_{\leq 0}$.   By the dimension assumption, as remarked at the beginning of Section \ref{sec2}, the latter contains at most $d+1$ elements of $A$.  Thus, $b\in A^P_{[d]}$.  The theorem is proved.

\end{document}